\documentclass[12pt]{article}
\usepackage{graphicx}
\usepackage{amsmath,amsthm,amssymb,enumerate}%, esint}
\usepackage{euscript,mathrsfs}
\usepackage[left=2cm,right=2cm,top=3.5cm,bottom=3.5cm]{geometry}
\usepackage{color}

\usepackage{soul}

\catcode`\@=11 \@addtoreset{equation}{section}

\catcode`\@=12

\allowdisplaybreaks

\newtheorem{Theorem}{Theorem}[section]
\newtheorem{Proposition}[Theorem]{Proposition}
\newtheorem{Lemma}[Theorem]{Lemma}
\newtheorem{Corollary}[Theorem]{Corollary}

\theoremstyle{definition}
\newtheorem{Definition}[Theorem]{Definition}

\newtheorem{Remark}[Theorem]{Remark}

\newcommand{\bTheorem}[1]{
\begin{Theorem} \label{T#1} }
\newcommand{\eT}{\end{Theorem}}

\newcommand{\bProposition}[1]{
\begin{Proposition} \label{P#1}}
\newcommand{\eP}{\end{Proposition}}

\newcommand{\bLemma}[1]{
\begin{Lemma} \label{L#1} }
\newcommand{\eL}{\end{Lemma}}

\newcommand{\bCorollary}[1]{
\begin{Corollary} \label{C#1} }
\newcommand{\eC}{\end{Corollary}}
\newcommand{\tvS}{\widetilde{S}}
\newcommand{\bRemark}[1]{
\begin{Remark} \label{R#1} }
\newcommand{\eR}{\end{Remark}}

\newcommand{\bDefinition}[1]{
\begin{Definition} \label{D#1} }
\newcommand{\eD}{\end{Definition}}

\newcommand{\bu}{\mathbf u}

\newcommand{\tmmathbf}[1]{\ensuremath{\boldsymbol{#1}}}

\newcommand{\tvm}{\tilde{\vc{m}}}
\newcommand{\bfphi}{\boldsymbol{\varphi}}

\newcommand{\bFormula}[1]{
\begin{equation} \label{#1}}
\newcommand{\eF}{\end{equation}}

\newcommand{\Ov}[1]{\overline{#1}}

\newcommand{\DC}{C^\infty_c}

\newcommand{\aleq}{\stackrel{<}{\sim}}

\newcommand{\ageq}{\stackrel{>}{\sim}}

\newcommand{\vr}{\varrho}

\newcommand{\tvr}{\tilde \vr}

\newcommand{\vt}{\vartheta}
\newcommand{\vu}{\vc{u}}
\newcommand{\vm}{\vc{m}}

\newcommand{\vc}[1]{{\bf #1}}

\newcommand{\Div}{{\rm div}_x}
\newcommand{\Grad}{\nabla_x}

\newcommand{\dx}{\,{\rm d} {x}}

\newcommand{\dt}{\,{\rm d} t }

\newcommand{\vU}{\vc{U}}

\newcommand{\intO}[1]{\int_{\Omega} #1 \ \dx}

\newcommand{\intRd}[1]{\int_{R^d} #1 \ \dx}

\newcommand{\D}{{\rm d}}

\newcommand{\mathd}{\mathrm{d}}

\def\softd{{\leavevmode\setbox1=\hbox{d}%
          \hbox to 1.05\wd1{d\kern-0.4ex{\char039}\hss}}}
\definecolor{Cgrey}{rgb}{0.85,0.85,0.85}
\definecolor{Cblue}{rgb}{0.50,0.85,0.85}
\definecolor{Cred}{rgb}{1,0,0}
\definecolor{fancy}{rgb}{0.10,0.85,0.10}

\newcommand\Cbox[2]{%
    \newbox\contentbox%
    \newbox\bkgdbox%
    \setbox\contentbox\hbox to \hsize{%
        \vtop{
            \kern\columnsep
            \hbox to \hsize{%
                \kern\columnsep%
                \advance\hsize by -2\columnsep%
                \setlength{\textwidth}{\hsize}%
                \vbox{
                    \parskip=\baselineskip
                    \parindent=0bp
                    #2
                }%
                \kern\columnsep%
            }%
            \kern\columnsep%
        }%
    }%
    \setbox\bkgdbox\vbox{
        \color{#1}
        \hrule width  \wd\contentbox %
               height \ht\contentbox %
               depth  \dp\contentbox
        \color{black}
    }%
    \wd\bkgdbox=0bp%
    \vbox{\hbox to \hsize{\box\bkgdbox\box\contentbox}}%
    \vskip\baselineskip%
}

\date{}

\begin{document}

%%%%%%%%%%%%%%%%%%%%%%%%%%%%%%%%

\title{On convergence of approximate solutions to the compressible Euler system}

\author{Eduard Feireisl
\thanks{The research of E.F. leading to these results has received funding from
the Czech Sciences Foundation (GA\v CR), Grant Agreement
18--05974S. The Institute of Mathematics of the Academy of Sciences of
the Czech Republic is supported by RVO:67985840. The stay of E.F. at TU Berlin is supported by Einstein Foundation, Berlin.}
\and Martina Hofmanov\' a \thanks{M.H. gratefully acknowledges the financial support by the German Science Foundation DFG via the Collaborative Research Center SFB1283.}
}

%\date{\today}

\maketitle

\centerline{Institute of Mathematics of the Academy of Sciences of the Czech Republic}
\centerline{\v Zitn\' a 25, CZ-115 67 Praha 1, Czech Republic}

\centerline{and}

\centerline{Institute of Mathematics, Technische Universit\"{a}t Berlin,}
\centerline{Stra{\ss}e des 17. Juni 136, 10623 Berlin, Germany}
\centerline{feireisl@math.cas.cz}

\bigskip

\centerline{Fakult\"at f\"ur Mathematik, Universit\"at Bielefeld}
\centerline{D-33501 Bielefeld, Germany}
\centerline{hofmanova@math.uni-bielefeld.de}

\begin{abstract}

We consider a sequence of approximate solutions to the compressible Euler system
admitting uniform energy bounds and/or satisfying 
the relevant field equations modulo an error vanishing in the asymptotic limit.
We show that such a sequence 
either {\bf (i)} converges strongly in the energy norm, or {\bf (ii)} the limit is not a weak solution of the associated Euler system.
This  is in sharp contrast to the incompressible case, where (oscillatory) approximate solutions may converge weakly to  solutions of the Euler system. Our approach leans on identifying a system of differential equations
satisfied by the associated turbulent defect measure and showing that it only has a trivial solution.

\end{abstract}


%{\bf MSC:}
\bigskip

\section{Introduction}
\label{i}

In \cite[Section 4]{GreTho}, Greengard and Thomann constructed a sequence $\{ \vc{v}_n \}_{n=1}^\infty$ of exact solutions to the \emph{incompressible} Euler system in $R^2$,
compactly supported in the space variable, and converging \emph{weakly} to the velocity field $\vc{v} = 0$. As $\vc{v} = 0$ is obviously
a solution of the Euler system, this is an example of a sequence of solutions to the incompressible Euler system defined on the whole space $R^2$
and converging weakly to another solution of the same problem. We show that such a scenario is impossible in the context of
\emph{compressible} fluid flows.

We consider consider a sequence of \emph{approximate} solutions to the compressible Euler system. Motivated by the numerical terminology 
we distinguish {\bf (i)} \emph{stable approximation}, where the approximate solutions satisfy the relevant uniform bounds, 
and {\bf (ii)} \emph{consistent approximation}, where the field equations of the Euler system are satisfied modulo an error vanishing in the asymptotic limit.
A prominent example of consistent approximation is the \emph{vanishing viscosity} limit, where the approximate solutions 
satisfy the Navier--Stokes system. In the light of the recent  results \cite{Chiod, ChiDelKre, ChiKre, ChKrMaSwI} indicating essential ill--posedness of the compressible
Euler system, the vanishing viscosity limit might be seen as a sound selection criterion to identify the physically relevant solutions of systems describing inviscid fluids, although this can be still arguable in view of the examples collected in the recent survey by Buckmaster and Vicol \cite{BucVic} and  Constantin and Vicol \cite{CoVi18}.
The principal difficulties of this process, caused in particular by the presence of kinematic boundaries, are well understood in the case of incompressible fluids, see e.g. the survey of E \cite{E1}. However, much less is known in the compressible case. Leaving apart the boundary layer issue, Sueur \cite{Sue1} proved unconditional convergence
in the barotropic case provided
the Euler system admits a smooth solution. A similar result was obtained for the full Navier--Stokes/Euler systems in \cite{Fei2015A}.
However, as many solutions of the Euler system are known to develop discontinuities in finite time, it is of essential interest to understand
the inviscid limit provided the target solution is not smooth. Very recently, Basari\' c \cite{Basa} identified the vanishing viscosity limit
with a measure--valued solution to the Euler system on general, possibly unbounded, spatial domains, which can be seen as a ``compressible'' counterpart of the pioneering work of DiPerna and Majda \cite{DIMA} in the incompressible case. The incompressible setting was further studied in  space dimension two and for vortex sheet initial data by DiPerna and Majda \cite{DipMaj88, DipMaj87} and Greengard and Thomann \cite{GreTho}. Their results show that the set, where the approximate solutions do not converge strongly is either
empty or its projection on the time axis is of positive measure.

As the name suggests, numerous consistent approximations can be identified 
with sequences of numerical solutions, see e.g. \cite{FeiLM}, \cite{FeiLMMiz}.
There is a strong piece of evidence, see e.g. Fjordholm et al. \cite{FjKaMiTa}, \cite{FjLyMiWe}, \cite{FjMiTa1}, that the numerical solutions to the compressible Euler system develop fast oscillations (wiggles) in the asymptotic limit.
The resulting object is described by the associated Young measure and it is therefore of interest to know in which sense the limit
Euler system is satisfied. In accordance with the seminal paper by DiPerna and Majda \cite{DIMA}, the limit should be identified with
a generalized \emph{measure--valued} solution of the Euler system. The concept of measure--valued solution used
also more recently in Basari\' c \cite{Basa}, however, follows the philosophy: the more general the better, while preserving a suitable weak (measure--valued)/strong uniqueness principle. Such an approach is typically beneficial for a number of applications in numerical analysis.
As a matter of fact, a more refined description of the asymptotic limit can be obtained via  Alibert--Bouchitt\'e's
\cite{AliBou} framework employed by Gwiazda, \'Swierczewska--Gwiazda, and Wiedemann \cite{GSWW}. Here, similarly to the work by Chen and Glimm \cite{CheGli}, the
measure--valued solutions are defined for the density $\vr$ and the weighted velocity $\sqrt{\vr} \vu$ yielding a rather awkward definition of a solution.

Our approach is based on estimating the distance between an approximate sequence and its limits by means of 
the so--called \emph{Bregman divergence}
\begin{equation} \label{BD}
\mathcal{E}\left( \vU \ \Big| \vc{V} \right) = \intO{ \Big[ E(\vU) - \xi \cdot (\vU - \vc{V}) - E(\vc{V}) \Big] } ,\ \xi \in \partial E(\vc{V}),
\end{equation}
where $\vU$, $\vc{V}$ are measurable functions on the fluid domain $\Omega \subset R^d$ ranging in
$R^m$, and $E: R^m \to [0, \infty]$ is a strictly convex function, see e.g. Sprung \cite{Sprung}. In the context of the Euler system, 
the function $E$ is the total energy; whence $\mathcal{E}$ may be see as \emph{relative energy} in the sense of Dafermos \cite{Daf4}.
Strict convexity of $E$ is then nothing other than a formulation of the principle of thermodynamic stability, where the relevant 
phase variables are the density $\vr$, the momentum $\vm$, and the total entropy $S$, cf. Bechtel, Rooney, and Forrest \cite{BEROFO}.

We consider both the full Euler system and its isentropic variant. In the former case, we show that any \emph{stable} approximation 
either converge pointwise or its limit is not a weak solution of the Euler system. The proof is based mainly on the fact that the 
total energy is a conserved quantity for the limit system. The isentropic case is more delicate, as the energy conservation 
is in general violated by the weak solutions. Here, we consider \emph{consistent} approximation and show that the energy defect, 
expressed through the asymptotic limit of the Bregman distance  is  
intimately related to \emph{turbulent defect measure} in the momentum equation. In fact, the defect in the momentum equation directly controls the defect in the energy. 
(The converse, meaning the defect in the energy controls the defect in the momentum equation, is also true and indispensable but not of direct use in the present setting). Furthermore, the turbulent defect measure $\mathbb{D}(t)$ is for a.e. time given by a (symmetric) positive semidefinite matrix--valued finite Borel  measure on the
physical space $\Omega \subset R^{d}$ in the sense that
\begin{equation*}%\label{eq:eq2}
\mathbb{D}(t) : (\xi \otimes \xi) \ \mbox{is a non--negative finite measure on}\ \Omega
\ \mbox{for any} \ \xi \in R^d,
\end{equation*}
and it can be identified along with a system of differential equations it obeys.
In particular, we show below that
the problem of convergence towards a weak solution reduces to solving a system of differential equations
\begin{equation} \label{def1}
\Div \mathbb{D}(t) = 0.
\end{equation}

The paper is organized as follows. In Section \ref{D}, we recall the concept of \mbox{weak solution} for both the complete Euler system
and its isentropic variant. We introduce the notion of 
stable and consistent approximations and state the main results. In Section \ref{C}, we study convergence of stable approximations 
to the full Euler system. Section \ref{S} is devoted to the same problem for the isentropic Euler system. Possible extensions of the results are discussed in Section \ref{DD}.

\section{Approximate solutions to the Euler system, main results}
\label{D}

The complete Euler system governing the time evolution of the density $\vr = \vr(t,x)$, the momentum $\vm = \vm(t,x)$, and the
energy $E = E(t,x)$ of a compressible perfect fluid reads:
\begin{equation} \label{D1}
\begin{split}
\partial_t \vr + \Div \vm &=0,\\
\partial_t \vm + \Div \left( \frac{\vm \otimes \vm}{\vr} \right) + \Grad p &= 0,\\
\partial_t E + \Div \left[ \left( E + p \right) \frac{\vm}{\vr} \right] &=0.
\end{split}
\end{equation}
We suppose the fluid is confined to a domain $\Omega \subset R^d$ with impermeable boundary,
\begin{equation} \label{D2}
\vm \cdot \vc{n}|_{\partial \Omega} = 0.
\end{equation}
Mostly we deal with the \emph{admissible} weak solutions satisfying the Euler system \eqref{D1} in the sense of distributions, together
with the (renormalized) entropy inequality
\begin{equation} \label{D3}
\partial_t (\vr Z(s)) + \Div \left( Z(s) \vm \right) \geq 0
\end{equation}
for any $Z \in BC(R)$, $Z' \geq 0$, cf. e.g. Chen and Frid \cite{CheFr1}. Here $p$ is the pressure and $s$ is the entropy related to the internal energy $e$ through
Gibbs' equation
\begin{equation} \label{D4}
\vt D s = De + p D \left( \frac{1}{\vr} \right),\ \mbox{where}\ \vt \ \mbox{is the absolute temperature.}
\end{equation}
Introducing the total entropy $S = \vr s$, we write all thermodynamic functions in terms of the basic \emph{phase variables}
$[\vr, \vm, S]$,
\[
E = \frac{1}{2} \frac{|\vm|^2}{\vr} + \vr e(\vr, S),\
p = p(\vr, S).
\]
The cornerstone of the forthcoming analysis is the \emph{thermodynamic stability hypothesis}:
\[
\mbox{The total energy}\
[\vr, \vm, S] \in R^{d+2} \mapsto E(\vr, \vm, S) \equiv \frac{1}{2} \frac{|\vm|^2}{\vr} + \vr e(\vr, S) \in [0, \infty]
\]
is a strictly convex l.s.c. function, 
where we set
\begin{equation} \label{D4a}
E(\vr, \vm, S) =  \infty \ \mbox{whenever}\ \vr < 0, \
E(0, \vm, S) = \lim_{\vr \to 0+} E(\vr, \vm, S),
\end{equation}
cf. Bechtel, Rooney, Forrest \cite{BEROFO}.
To avoid further technicalities, we suppose the polytropic relation between the pressure and the internal energy
\[
p = (\gamma - 1) \vr e,\ \gamma > 1, \ \mbox{and set}\ e = c_v \vt,\ c_v = \frac{1}{\gamma - 1}.
\]
Accordingly, the total energy takes the form
\begin{equation} \label{D5}
E(\vr, \vm, S) = \left\{ \begin{array}{l} \frac{1}{2} \frac{|\vm|^2}{\vr} + \vr^\gamma \exp \left( \frac{S}{c_v \vr} \right) \ \mbox{if}\ \vr > 0,\\
0 \ \mbox{for}\ \vr = 0, \ \vm = 0,\ S \leq 0,\\ \infty, \ \mbox{otherwise} \end{array} \right.
\end{equation}
for which the desired convexity has been verified in \cite{BreFeiHof19C}.

\subsection{Weak solutions to the complete Euler system}

\begin{Definition}[Admissible weak solution to complete Euler system] \label{DD1}

Let $\Omega \subset R^d$, $d=1,2,3$ be a domain with Lipschitz boundary.

We say that $[\vr, \vm, S]$ is an \emph{admissible weak solution} to the Euler system \eqref{D1}--\eqref{D3} in $(0,T) \times \Omega$
with the initial data $[\vr_0, \vm_0, S_0]$, if
\begin{itemize}
\item
\[
\vr \geq 0 \ \mbox{a.a. in}\ (0,T) \times \Omega,\ S = 0 \ \mbox{a.a. in the set}\
\{ \vr = 0 \};
\]
\item
\begin{equation} \label{D6}
\left[ \intO{ \vr \varphi } \right]_{t = 0}^{t = \tau}
= \int_0^\tau \intO{ \left[ \vr \partial_t \varphi + \vm \cdot \Grad \varphi \right] },\ \vr(0, \cdot) = \vr_0,
\end{equation}
for any $0 \leq \tau < T$, $\varphi \in C^1_c([0,T) \times \Ov{\Omega})$;
\item
\begin{equation} \label{D7}
\begin{split}
\left[ \intO{ \vm \cdot \bfphi } \right]_{t = 0}^{t = \tau}
&= \int_0^\tau \intO{ \left[ \vm \cdot \partial_t \bfphi + 1_{\vr > 0} \frac{\vm \otimes \vm}{\vr} : \Grad \bfphi
+ p(\vr, S) \Div \bfphi \right] }\dt,\\ \vm(0, \cdot) &= \vm_0,
\end{split}
\end{equation}
for any $0 \leq \tau < T$, $\bfphi \in C^1_c([0,T) \times \Ov{\Omega})$, $\bfphi \cdot \vc{n}|_{\partial \Omega} = 0$;
\item
\begin{equation} \label{D8}
\begin{split}
\left[ \intO{ E(\vr, \vm, S) \varphi } \right]_{t = 0}^{t = \tau}&\\
= \int_0^\tau & \intO{ \left[ E(\vr, \vm, S) \partial_t \varphi + 1_{\vr > 0}
\left[ \left( E(\vr, \vm, S) + p(\vr, S)  \right) \frac{\vm}{\vr} \right] \cdot \Grad \varphi \right]
}\dt,\\ E(\vr, \vm, S)(0, \cdot) &= E(\vr_0, \vm_0, S_0),
\end{split}
\end{equation}
for any $0 \leq \tau < T$, $\varphi \in C^1_c([0,T) \times \Ov{\Omega})$;
\item
\begin{equation} \label{D9}
\begin{split}
\left[ \intO{ \vr Z \left(\frac{S}{\vr} \right) \varphi } \right]_{t = 0}^{t = \tau}
&\geq \int_0^\tau \intO{ \left[ \vr Z\left(\frac{S}{\vr} \right) \partial_t \varphi +
Z \left( \frac{S}{\vr} \right) {\vm} \cdot \Grad \varphi \right] }\dt,\\ 
\vr Z \left(\frac{S}{\vr} \right)(0, \cdot) &= \vr_0 Z \left(\frac{S_0}{\vr_0} \right),
\end{split}
\end{equation}
for a.a. $0 \leq \tau < T$, and 
any $\varphi \in C^1_c([0,T) \times \Ov{\Omega})$, $\varphi \geq 0$, and $Z \in BC(R) \cap C^1(R)$,
$Z' \geq 0$.

\end{itemize}

\end{Definition}

In Definition \ref{DD1}, we tacitly assume that all quantities under integrals are at least locally integrable
in $[0,T) \times \Ov{\Omega}$.

\subsection{Weak solutions to the isentropic Euler system}

The \emph{isentropic} Euler system is formally obtained from \eqref{D1} by requiring the entropy $s = \Ov{s}$ to be constant.
The total energy given by \eqref{D5} simplifies to 
\begin{equation} \label{D10}
E = E(\vr, \vm) =  \frac{1}{2} \frac{|\vm|^2}{\vr} + P(\vr),\ P(\vr) \equiv \frac{a}{\gamma - 1} \vr^\gamma, \
p = p(\vr) = a \vr^\gamma, \ a > 0.
\end{equation}
We consider the isentropic Euler system on the whole space $R^d$, with the far field boundary conditions
\begin{equation} \label{D11}
\vr \to \vr_\infty \geq 0,\ \vm \to \vm_\infty = \vr_\infty \vu_\infty \ \mbox{as}\ |x| \to \infty,
\end{equation}
where $\vr_\infty$ and $\vu_\infty$ are give constant fields.
Consequently, it is more convenient to replace $E$ by the relative energy
\[
\begin{split}
E &\left( \vr, \vm \ \Big|\ \vr_\infty, \vm_\infty\right) \\ &=
\frac{1}{2} \frac{|\vm|^2}{\vr} - \vm \cdot \vu_\infty + \frac{1}{2} \vr |\vu_\infty|^2 +
P(\vr) - P'(\vr_\infty) (\vr - \vr_\infty) - P(\vr_\infty) \\ &=
\frac{1}{2} \vr \left| \frac{\vm}{\vr} - \vu_\infty \right|^2 + P(\vr) - P'(\vr_\infty) (\vr - \vr_\infty) - P(\vr_\infty).
\end{split}
\]
As pointed out in the introductory part, the relative energy is nothing other than the Bregman divergence associated to the convex function $E$, cf. \eqref{BD}.

\begin{Definition}[Weak solution to isentropic Euler system] \label{DD2}

We say that $[\vr, \vm]$ is a \emph{weak solution} to the Euler system  in $(0,T) \times R^d$,
with the initial data $[\vr_0, \vm_0]$
and the far field conditions \eqref{D11},
if
\begin{itemize}
\item
\[
\vr \geq 0 \ \mbox{a.a. in}\ (0,T) \times R^d;
\]
\item
\begin{equation} \label{D12}
\left[ \intRd{ \vr \varphi } \right]_{t = 0}^{t = \tau}
= \int_0^\tau \intRd{ \left[ \vr \partial_t \varphi + \vm \cdot \Grad \varphi \right] },\ \vr(0, \cdot) = \vr_0,
\end{equation}
for any $0 \leq \tau < T$, $\varphi \in C^1_c([0,T) \times R^d)$;
\item
\begin{equation} \label{D13}
\begin{split}
\left[ \intRd{ \vm \cdot \bfphi } \right]_{t = 0}^{t = \tau}
&= \int_0^\tau \intRd{ \left[ \vm \cdot \partial_t \bfphi + 1_{\vr > 0} \frac{\vm \otimes \vm}{\vr} : \Grad \bfphi
+ p(\vr) \Div \bfphi \right] }\dt, \\ \vm(0, \cdot) = \vm_0,
\end{split}
\end{equation}
for any $0 \leq \tau < T$, $\bfphi \in C^1_c([0,T) \times R^d, R^d)$.

\end{itemize}

We say that a weak solution is \emph{admissible}, if, in addition, the energy inequality
\begin{equation} \label{D14}
\intRd{ E \left(\vr, \vm \ \Big| \ \vr_\infty, \vm_\infty \right)(\tau, \cdot) } 
\leq \intRd{ E \left(\vr_0, \vm_0 \ \Big| \ \vr_\infty, \vm_\infty \right) } 
\end{equation}
holds for any $0 \leq \tau < T$.

\end{Definition}

Note that the total energy balance \eqref{D8} that 
is an \emph{integral part} of the weak formulation for the complete Euler system has been replaced by the integrated energy inequality
\eqref{D14} that plays the role of \emph{admissibility condition} similar to the entropy inequality \eqref{D9}.
In \eqref{D14}, we tacitly assume that the initial (relative) energy is finite, meaning that the initial data satisfy the far field
conditions \eqref{D11}.

\subsection{Stable and consistent approximations}

The following two definitions are motivated by the terminology used in the \emph{numerical analysis}.

\begin{Definition}[Stable approximation of the full Euler system] \label{DD3}

We say that a sequence $$\{ \vr_n, \vm_n, S_n \}_{n=1}^\infty$$ is a \emph{stable approximation} of the
full Euler system in $(0,T) \times \Omega$, with the initial data $[\vr_0, \vm_0, S_0]$, if:
\begin{equation} \label{D15a}
\begin{split}
\vr_n \geq 0,\ {\rm ess} \sup_{\tau \in (0,T)} \intO{ \vr_n(\tau, \cdot)  } &\leq M,\\
{\rm ess} \inf_{\tau \in (0,T)} \intO{ S_n (\tau, \cdot) } &\geq \underline{S}
\end{split}
\end{equation}
uniformly for $n \to \infty$;
\begin{equation} \label{D15}
{\rm ess} \sup_{\tau \in (0,T)} \intO{ E(\vr_n, \vm_n, S_n) } \leq \intO{ E(\vr_0, \vm_0, S_0)} + e_n
\ \mbox{for all}\ n = 1,2,\dots
\end{equation}
where $e_n \to 0$ as $n \to \infty$.

\end{Definition}

Note that both \eqref{D15a} and \eqref{D15} obviously hold for any admissible weak solution in the sense of Definition \ref{DD1}.

Next, we introduce the concept of consistent approximation of the isentropic Euler system in $(0,T) \times R^d$ supplemented with
the far field conditions \eqref{D11}.

\begin{Definition}[Consistent approximation of isentropic Euler system] \label{DD4}

We say that a sequence $\{ \vr_n, \vm_n \}_{n=1}^\infty$ is a \emph{consistent approximation} of the isentropic Euler system
in $(0,T) \times R^d$, with the far field conditions \eqref{D11} if:
\begin{itemize}
\item
\[
\vr_n \geq 0 \ \mbox{a.a. in}\ (0,T) \times R^d;
\]
\item
\begin{equation} \label{D16}
\int_0^\tau \intRd{ \left[ \vr_n \partial_t \varphi + \vm_n \cdot \Grad \varphi \right] } = e^1_n(\tau, \varphi)
\end{equation}
for any $\varphi \in \DC((0,T) \times R^d)$, $0 \leq \tau \leq T$;
\item
\begin{equation} \label{D17}
\int_0^\tau \intRd{ \left[ \vm_n \cdot \partial_t \bfphi + 1_{\vr > 0} \frac{\vm_n \otimes \vm_n}{\vr_n} : \Grad \bfphi
+ p(\vr_n) \Div \bfphi \right] }\dt = e^2_n (\tau, \bfphi)
\end{equation}
for any $\bfphi \in \DC((0,T) \times R^d, R^d)$, $0 \leq \tau \leq T$;
\item
\begin{equation} \label{D18}
\intRd{ E \left(\vr_n, \vm_n \ \Big| \ \vr_\infty, \vm_\infty \right) (\tau, \cdot) } \leq c
\end{equation}
uniformly for $0 \leq \tau \leq T$, $n=1,2, \dots$;
\item
\begin{equation} \label{D19}
e^1_n (\tau, \varphi) \to 0,\ e^2_n (\tau, \bfphi) \to 0 \ \mbox{as}\ n \to \infty
\end{equation}
for any fixed $0 \leq \tau \leq T$, $\varphi \in \DC((0,T) \times R^d)$, $\bfphi \in \DC((0,T) \times R^d; R^d)$.

\end{itemize}

\end{Definition}

Note carefully the difference between stable and consistent approximation. Stable approximation only satisfies the relevant
{\it a priori} bounds and approaches the energy of the initial data in the asymptotic limit. Consistent approximation satisfies the weak formulation of the
field equations modulo a small error vanishing in the asymptotic limit.

\subsection{Main results}

We start by the result concerning stable approximation to the complete Euler system. Recall that the only uniform bounds available
result from the hypothesis \eqref{D15a}, and the energy inequality \eqref{D15}. 
In particular, as we shall see below, the uniform bounds \eqref{D15a}, \eqref{D15} guarantee only $L^1-$integrability of the phase variables
$(\vr_n, \vm_n, S_n)$ with respect to the $x-$variable. Accordingly, we consider the concept of
\emph{biting limit} in the sense of Ball and Murat \cite{BAMU} to describe the asymptotic behavior of a stable approximation to the complete Euler system. The result reads as follows.

\begin{Theorem}[Asymptotic limit of stable approximation] \label{DT1}
Let $\Omega \subset R^d$ be a bounded Lipschitz domain. Let $\{ \vr_n, \vm_n, S_n \}_{n=1}^\infty$
be a stable approximation of the complete Euler system in the sense of Definition \ref{DD3}, with the initial data 
\begin{equation} \label{D20b}
\vr_0 > 0, \ \vm_0,\ S_0 \geq \vr_0 \underline{s}, \ \mbox{where}\ \underline{s} \in R.
\end{equation}

Then there exists a subsequence (not relabeled for simplicity) enjoying the following properties:
\begin{equation} \label{D20}
{\rm ess} \sup_{\tau \in (0,T)} \left[ \| \vr_n(\tau, \cdot) \|_{L^1(\Omega)} +
\| \vm_n (\tau, \cdot) \|_{L^1(\Omega; R^d)} + \| S_n(\tau, \cdot) \|_{L^1(\Omega)} \right] \leq c;
\end{equation}
the sequence $\{ \vr_n, \vm_n, S_n \}_{n=1}^\infty$ admits a biting limit $[\vr, \vm, S]$,
\[
[\vr, \vm, S] \in L^\infty(0,T; L^1(\Omega; R^{d + 2})).
\]
If, moreover, $[\vr, \vm, S]$ is an admissible weak solution to the complete Euler system
specified in Definition \ref{DD1}, then
\[
\vr_n \to \vr, \ \vm_n \to \vm,\ S_n \to S \ \mbox{a.a. in}\ (0,T) \times \Omega.
\]

\end{Theorem}

Our second result concerns the asymptotic behavior of a consistent approximation to the isentropic Euler system on $R^d$.

\begin{Theorem}[Asymptotic limit of consistent approximation] \label{DT2}
Let $\{ \vr_n, \vm_n \}_{n=1}^\infty$
be a consistent approximation of the isentropic Euler system in $(0,T) \times R^d$ in the sense of Definition \ref{DD4}.

Then there exists a subsequence (not relabeled for simplicity) enjoying the following properties:
\begin{equation} \label{D21}
\begin{split}
(\vr_n - \vr) &\to 0 \ \mbox{weakly-(*) in}\ L^\infty(0,T; L^\gamma + L^2 (R^d)),\\
(\vm_n - \vm) &\to 0 \ \mbox{weakly-(*) in}\ L^\infty(0,T; L^{\frac{2 \gamma}{\gamma + 1}} + L^2(\Omega)),
\end{split}
\end{equation}
where
\[
{\rm ess} \sup_{\tau \in (0,T)} {E} \left( \vr, \vm \ \Big| \ \vr_\infty, \vm_\infty \right) < \infty.
\]
If, moreover, $[\vr, \vm]$ is a weak solution to the isentropic Euler system
specified in Definition \ref{DD2}, then
\[
\int_0^T \intO{ {E} \left(\vr_n, \vm_n \ \Big| \vr , \vm \right) } \dt \to 0,
\]
in particular,
\[
\begin{split}
(\vr_n - \vr) &\to 0 \ \mbox{in}\ L^q(0,T; L^\gamma + L^2 (R^d)),\\
(\vm_n - \vm) &\to 0 \ \mbox{in}\ L^q(0,T; L^{\frac{2 \gamma}{\gamma + 1}} + L^2(\Omega)),
\end{split}
\]
for any $1 \leq q < \infty$. Thus, for a suitable subsequence,
\[
\vr_n \to \vr, \ \vm_n \to \vm \ \mbox{a.a. in}\ (0,T) \times \Omega.
\]
\end{Theorem}

We point out that the results stated in Theorems \ref{DT1}, \ref{DT2} require extracting a suitable subsequence. In both cases,
the convergence is necessarily strong (pointwise a.a.) as soon as the limit is an admissible weak solution to the
system.

\section{Convergence of stable approximations to the full Euler system}
\label{C}

Our goal is to prove Theorem \ref{DT1}. We start by establishing uniform bounds for the stable approximation. 

\subsection{Uniform bounds}

We establish the uniform bounds claimed in \eqref{D20}. To see this, we choose an arbitrary point $[\tvr, 0, \tvS] \in R^{d + 2}$, $\tvr > 0$, and consider
the quantity
\[
\begin{split}
0 &\leq
E(\vr_n, \vm_n, S_n) - \frac{\partial E(\tvr, 0, \tvS)}{\partial \vr}
(\vr_n - \tvr) - \frac{\partial E(\tvr, 0, \tvS)}{\partial \vm} \cdot (\vm_n - \tvm) -
\frac{\partial E(\tvr, 0, \tvS)}{\partial S}  (S_n - \tvS) \\ &-
E(\tvr, 0, \tvS) = \frac{1}{2} \frac{|\vm_n|^2}{\vr_n} + \vr_n e(\vr_n, S_n)
- \frac{\partial (\vr e) (\tvr,\tvS)}{\partial \vr}(\vr - \tvr) -
\frac{\partial (\vr e) (\tvr,\tvS)}{\partial S}(S - \tvS) - \tvr e(\tvr, \tvS).
\end{split}
\]
Seeing that $\frac{\partial E}{\partial S} = \vt > 0$, we conclude
\[
\begin{split}
\int_\Omega &\Big[ E(\vr_n, \vm_n, S_n) - \frac{\partial E(\tvr, 0, \tvS)}{\partial \vr}
(\vr_n - \tvr) - \frac{\partial E(\tvr, 0, \tvS)}{\partial \vm} \cdot (\vm_n - \tvm) -
\frac{\partial E(\tvr, 0, \tvS)}{\partial S} (S_n - \tvS)
\\
&- E(\tvr, 0, \tvS) \Big] \dx \leq c(\tvr,  \tvS) \left(1 + \intO{E(\vr_n, \vm_n, S_n)} + \intO{ \vr_n } -
\intO{ S_n }    \right) \\ &\leq c(\tvr,  \tvS) \left(1 + \intO{E(\vr_{0}, \vm_{0}, S_{0})} + M -
\underline{S} +e_n \right)
\end{split}
\]
As $E$ is strictly convex at $[\tvr, 0, \tvS]$, we have
\[
\begin{split}
E(\vr_n, \vm_n, S_n) - \frac{\partial E(\tvr, 0, \tvS)}{\partial \vr}
(\vr_n - \tvr) &- \frac{\partial E(\tvr, 0, \tvS)}{\partial \vm} \cdot (\vm_n - \tvm) -
\frac{\partial E(\tvr, 0, \tvS)}{\partial S}  (S_n - \tvS) \\ &-
E(\tvr, 0, \tvS) \ageq |\vr_n - \tvr| + |\vm_n | + |S_n - \tvS|
\end{split}
\]
as soon as
\[
|\vr_n - \tvr| + |\vm_n | + |S_n - \tvS| \geq 1.
\]
Since $\Omega$ is bounded, the estimates \eqref{D20} follow.

\subsection{Strong convergence}

We shall systematically extract various subsequence keeping the labeling of the original sequence. In view of \eqref{D15}, \eqref{D20}, 
the sequence $\{ \vr_n, \vm_n, S_n \}_{n = 1}^\infty$ generates a Young measure 
\[
\mathcal{V} \in L^\infty_{{\rm weak}-(*)} ((0,T) \times \Omega; \mathcal{P}(R^{d + 2})), \ R^{d + 2} = \left\{ (\tvr, \tvm, \tvS) \in R^{d + 2} \right\}.
\] 
Moreover, $\mathcal{V}_{t,x}$ possesses finite first moments for a.a. $(t,x)$ and we can set 
\[
\vr(t,x) = \left< \mathcal{V}_{t,x}; \tvr \right>, \ \vm(t,x) = \left< \mathcal{V}_{t,x}; \tvm \right>,\ S(t,x) = \left< \mathcal{V}_{t,x}; \tvS \right>. 
\]
As observed by Ball and Murat \cite{BAMU}, the trio $[\vr, \vm, S]$ corresponds to the biting limit of the sequence $\{ \vr_n, \vm_n, S_n \}_{n = 1}^\infty$. 
Finally, in view of the energy bound \eqref{D15}, we have 
\[
E(\vr_n, \vm_n, S_n)  \to \Ov{E(\vr, \vm, S)} \ \mbox{weakly-(*) in}\ L^\infty_{w^*}(0,T; \mathcal{M}^+(\Ov{\Omega})),
\]
where the symbol $\mathcal{M}^+$ denotes the set of non--negative Borel measures. In view of the hypothesis \eqref{D15}, 
\begin{equation} \label{C1a}
\intO{ E(\vr_0, \vm_0, S_0) } \geq \int_{\Ov{\Omega}} \D \Ov{E(\vr, \vm, S)}(\tau, \cdot),
\end{equation}
and
\begin{equation} \label{C1} 
\Ov{E(\vr, \vm, S)}(\tau, \cdot) \geq \left< \mathcal{V}_{\tau, \cdot} ; E(\tvr, \tvm, \tvS) \right> \geq E(\vr, \vm, S)(\tau, \cdot) \ \mbox{for a.a.}\ \tau \in (0,T)
\end{equation}
in the sense of non--negative measures on $\Ov{\Omega}$. Note that the first inequality in 
\eqref{C1} follows from lower semi--continuity of the energy, while the second one 
follows from its convexity, see e.g. \cite[Section 3.2]{FeiLukMizSheWa}. In particular, the biting limit $[\vr, \vm, S]$ belongs to the class 
\[
[\vr, \vm, S] \in L^\infty(0,T; L^1(\Omega; R^{d + 2})).
\]

Finally, suppose that $[\vr, \vm, S]$ is an admissible weak solution of the Euler system in the sense of Definition \ref{DD1}. In particular, the total energy balance 
\eqref{D8} holds; whence 
\begin{equation} \label{C2}
\intO{ E(\vr, \vm, S) (\tau, \cdot) } = \intO{ E(\vr_0, \vm_0, S_0) } \ \mbox{for any}\ 0 \leq \tau \leq T.
\end{equation}
Moreover, as the entropy equation \eqref{D9} is satisfied in the renormalized sense, we can deduce from the hypothesis \eqref{D20b}
the entropy minimum principle,
\begin{equation} \label{C2a}
S(t,x) \geq \vr(t,x) \underline{s} \ \mbox{for a.a.}\ (t,x),
\end{equation}
see \cite{BreFei17}.

Going back to \eqref{C1} we conclude 
\begin{equation} \label{C3}
\begin{split}
\intO{ E(\vr_0, \vm_0, S_0) } &= \int_{\Ov{\Omega}} \D \Ov{E(\vr, \vm, S) }(\tau, \cdot), \\  
\Ov{E(\vr, \vm, S) } &= \left< \mathcal{V} ; E(\tvr, \tvm, \tvS) \right>  = 
E(\vr, \vm, S). 
\end{split}
\end{equation}
The second equality, specifically, 
\[
\Ov{E(\vr, \vm, S) }  = 
\left< \mathcal{V} ; E(\tvr, \tvm, \tvS) \right> 
\]
means that the concentration defect associated to the sequence $\{ E(\vr_n, \vm_n, S_n) \}_{n=1}^\infty$ vanishes, specifically, 
\[
E(\vr_n, \vm_n, S_n) \to \left< \mathcal{V} ; E(\tvr, \tvm, \tvS) \right> = 
E(\vr, \vm, S) \ \mbox{weakly in}\ L^1((0,T) \times \Omega), 
\]
cf. \cite{FeiLukMizSheWa}.

The third equality, together with \eqref{C2a}, implies the desired pointwise convergence. To see this, we need the following result that may be of independent interest.  

\begin{Lemma} [Sharp form of Jensen's inequality] \label{LK1}

Suppose that $E: R^m \to [0, \infty]$ is an l.s.c. convex function satisfying: 
\begin{itemize} 
\item 
$E$ is strictly convex on its domain of positivity, meaning for any $y_1, y_2 \in R^m$ such that 
$0 < E(y_1) < \infty$, $E(y_2) < \infty$, $y_1 \ne y_2$, we have 
\[
E \left( \frac{y_1 + y_2}{2} \right) < \frac{1}{2} E(y_1) + \frac{1}{2} E(y_2).
\]

\item If $y \in \partial {\rm Dom}[E]$, then either $E(y) = \infty$ or $E(y) = 0$, in other words, 
\begin{equation} \label{YK1}
E(y) =  0 \ \mbox{whenever}\ y \in {\rm Dom}[E] \cap \partial {\rm Dom}[E].
\end{equation}
Let $\nu \in \mathcal{P}[R^m]$ be a (Borel) probability measure with finite first moment satisfying 
\begin{equation} \label{YK2}
E (\left< \nu; \widetilde{y} \right>)  = \left< \nu ; E (\widetilde{y}) \right> < \infty. 
\end{equation}

Then 
{\bf (i)} either
\[
\nu = \delta_Y, \ Y = \left< \nu; \widetilde{y} \right> \in {\rm Dom}[E], \ E(Y) > 0,
\]

{\bf (ii)} or 
\[
{\rm supp} [\nu] \subset \left\{ y \in R^m \ \Big| \ E(y) = 0 \right\}.
\]

\end{itemize}

\end{Lemma}

\begin{proof}

First observe that, obviously, $\left< \nu; \widetilde{u} \right> \in {\rm Dom}[E]$, and, by virtue of \eqref{YK2} 
and positivity of $E$, 
\[
\nu \left\{ R^m \setminus {\rm Dom}[E] \right\} = 0.
\]  

{\bf (i)} Suppose first that $Y \equiv \left< \nu; \widetilde{y} \right> \in {\rm int}[ {\rm Dom}[E]]$, $E(Y) > 0$. Then there exists 
\[
\Lambda \in \partial E ( Y ) 
\]
such that 
\[
E(y) \geq E (Y) + \Lambda \cdot (y - Y) 
\ \mbox{for any}\ y \in R^m. 
\]
As $E$ is strictly convex in ${\rm Dom}[E] \cap \{ E > 0 \}$, however, we claim that the above inequality must be sharp:
\[
E(y) - E (Y) - \Lambda \cdot (y -  Y) > 0 
\ \mbox{for all} \ y \in R^d, \ y \ne Y. 
\] 
Now it follows from \eqref{YK2} that 
\[
\Big< \nu; E(\widetilde{y}) - E (Y) 
- \Lambda \cdot (\widetilde{y}  - Y )\Big> = 0
\]
which yields the desired conclusion {\bf (i)}. 

\medskip

{\bf (ii)} Suppose that $Y = \left< \nu; \widetilde{y} \right> \in {\rm Dom}[E] \cap \partial {\rm Dom}[E]$ 
or $E(Y) = 0$. In accordance with the hypothesis \eqref{YK1}, we have in both cases
\[
E(Y) = 0.
\]
Consequently, we get from \eqref{YK2}, 
\[
\left< \nu; E(\widetilde{y}) \right> = 0 
\] 
which implies that $\nu$ is supported by zero points of $E$ as $E \geq 0$ which is the alternative {\bf (ii)}.

\end{proof}

In accordance with \eqref{C3}, 
\[
\left< \mathcal{V} ; E(\tvr, \tvm, \tvS) \right> (t,x)  = 
E(\vr, \vm, S)(t,x)\ \mbox{for a.a.}\ (t,x).
\] 
Clearly, $E$ satisfies the hypotheses of Lemma \ref{LK1}; whence either 
$\mathcal{V}_{t,x}$ is a Dirac mass, specifically, 
\begin{equation} \label{C7}
\mathcal{V}_{t,x} = \delta_{\vr(t,x), \vm(t,x), S(t,x)},
\end{equation}
or 
\[
{\rm supp} [\mathcal{V}_{t,x} ] \subset \left\{ \tvr = 0, \ \tvm = 0, \ \tvS \leq 0 \right\},  
\] 
which, combined with \eqref{C2a}, yields again \eqref{C7}. Indeed \eqref{C2a} 
means that the barycenter of $\mathcal{V}_{t,x}$ is located above the line $\tvS = \tvr \underline{s}$.
As the Young measure is a Dirac mass, we conclude the 
sequence $\{ \vr_n, \vm_n, S_n \}_{n=1}^\infty$ converges in measure; whence a suitable subsequence converges a.a. 
We have proved Theorem \ref{DT1}.  

\section{Convergence of consistent approximations to the isentropic Euler system}
\label{S}

Our goal is to show Theorem \ref{DT2}. It turns out the proof is more complicated than that of Theorem \ref{DT1} as 
the weak solution satisfies merely the field equations \eqref{D12}, \eqref{D13}.

\subsection{Turbulent defect measures}
\label{A}

In the following, we pass several times to suitable subsequences in the vanishing viscosity sequence without explicit relabeling. However, it is easy to see that it is enough to show the conclusion of Theorem \ref{DT2} for a subsequence once  
the limit $[\vr,\vm]$ has been fixed.

It follows from the bounds imposed by the
energy inequality \eqref{D18} that we may suppose
\[
(\vr_n - {\vr}_\infty) \to (\vr - {\vr}_\infty) \ \mbox{weakly-(*) in}\ L^\infty (0,T; (L^\gamma + L^2) (R^d)),
\]
\begin{equation}\label{eq:M1}
(\vm_n - \vm_\infty) \to (\vm - \vm_\infty) \ \mbox{weakly-(*) in}\  L^\infty (0,T; (L^{\frac{2 \gamma}{\gamma + 1}}+L^{2})(R^d; R^d)).
\end{equation}
In particular, we get \eqref{D21}. Indeed,  
as the total energy is $E(\vr, \vm)$ is a strictly convex function of $(\vr, \vm)$, it is easy to check that
\begin{equation}\label{eq:eq}
\begin{split}
E(\vr, \vm|\vr_{\infty},\vm_{\infty})& \ageq  (\vr - {\vr}_\infty)^2 + (\vm - \vm_\infty)^2 \ 
\ \mbox{for}\ \frac12 \vr_\infty \leq \vr \leq 2 {\vr}_\infty, \ \ \frac12 |\vm_\infty| \leq |\vm| \leq 2 |{\vm}_\infty| , \\ 
&\ageq 1 + \vr^\gamma 
+ \frac{|\vm|^2}{\vr} \ \mbox{otherwise};
\end{split}
\end{equation}
whence the desired bounds follow from the energy inequality \eqref{D21}.

\subsubsection{Internal energy  and pressure defect}

Next,
recall that the sequence
\[
0 \leq
P(\vr_n) - P'({\vr}_\infty) (\vr_n - {\vr}_\infty) - P({\vr}_\infty),\ n=1,2,\dots,
\]
is bounded in $L^\infty(0,T; L^1 (R^d))$ uniformly in $n$ by \eqref{D21}.
It holds
$$L^\infty(0,T; L^1 (R^d))\subset L^{\infty}_{w^{*}} (0, T, \mathcal{M} (R^d)),$$ where the symbol $\mathcal{M}(R^{d})$ denotes the set of  finite Borel measures on $R^d$ and $L^{\infty}_{w^{*}}(0,T;\mathcal{M}(R^{d}))$ stands for the space of weak-(*)-measurable mappings $\nu:[0,T]\to\mathcal{M}(R^{d})$ such that
$$
{\rm ess} \sup_{\tau \in[0,T]}\|\nu (\tau) \|_{\mathcal{M}(R^{d})}<\infty.
$$
In addition, $L^{\infty}_{w^{*}} (0, T, \mathcal{M} (R^d))$ is the dual of $L^1 (0, T, C_0 (R^d))$ hence passing   to a suitable subsequence as the case may be, there is $\mathcal{P}\in L^\infty_{w^{*}}(0,T; \mathcal{M}(R^d))$ such that
\[
P(\vr_n) - P'({\vr}_\infty) (\vr_n - {\vr}_\infty) - P({\vr}_\infty) \to \mathcal{P}
\ \mbox{weakly-(*) in}\ L^\infty_{w^{*}}(0,T; \mathcal{M}(R^d)).
\]
As the function $P$ is convex and the approximate internal energies are non--negative, we  deduce by weak lower semicontinuity that
\[
\mathfrak{R}_e \equiv \mathcal{P} - \left[ P(\vr) - P'({\vr}_\infty) (\vr - {\vr}_\infty) - P({\vr}_\infty) \right]
\in L^\infty_{w^{*}}(0,T; \mathcal{M}^+(R^d)),
\]
where $\mathcal{M}^+(R^d)$ denotes the set of non--negative finite Borel measures on $R^{d}$. This defines  the internal energy defect measure $\mathfrak{R}_e$. It is important to note that
\begin{equation} \label{A1}
\begin{split}
\int_0^T \int_{R^{d}}\psi(t) \varphi(x)\, \D \mathfrak{R}_e(t) \dt  &= \lim_{n \to \infty}
\int_0^T \intO{ \psi(t) \varphi(x) \left( P(\vr_n) - P(\vr) \right) } \dt \\ &\mbox{for any}\
\psi \in L^1(0,T), \ \varphi \in C_c(R^d),
\end{split}
\end{equation}
which will be used later.

\subsubsection{Viscosity defect}

Writing
\[
\mathbb{C}_n \equiv  
{\bf 1}_{\vr_n > 0} \left[ \frac{\vm_n \otimes \vm_n}{\vr_n} - \bu_\infty \otimes \vm_n - \vm_n \otimes \bu_\infty + \vr_n \bu_\infty \otimes \bu_\infty \right]
\] 
we obtain the existence of $\mathbb{C} \in L^\infty_{w^{*}} (0,T; \mathcal{M}^+(R^d; R^{d \times d}_{\rm sym}))$, where  $\mathcal{M}^+(R^d;R^{d \times d}_{\rm sym})$ is the set of finite symmetric positive semidefinite matrix--valued (signed) Borel measures, such that
\[
\mathbb{C}_n \to \mathbb{C} \ \mbox{weakly-(*) in}\
L^\infty_{w^{*}} (0,T; \mathcal{M}^+(R^d; R^{d \times d}_{\rm sym})).
\]
More specifically,
each component $C_{i,j}$ is a finite signed measure on $R^d$, $C_{i,j} = C_{j,i}$, and
\begin{equation} \label{eq:eq2}
\mathbb{C}(t) : (\xi \otimes \xi) \in \mathcal{M}^+(R^d) \ \mbox{for any}\ \xi \in R^d
\ \mbox{and a.a.}\ t \in (0,T).
\end{equation}

The viscosity defect measure is then defined by
\[
\mathfrak{R}_v \equiv \mathbb{C} - {\bf 1}_{\vr > 0} \left[ \frac{ \vm \otimes \vm }{\vr} - \bu_\infty \otimes \vm - \vm \otimes \bu_\infty + \vr \bu_\infty \otimes \bu_\infty\right] \in
L^\infty_{w^{*}}(0,T; \mathcal{M}(R^d; R^{d \times d}_{\rm sym})).
\]
Now, a simple but crucial observation is that the $\mathfrak{R}_v$ is positive semidefinite. To see this, we compute
\[
\begin{split}
\mathfrak{R}_v : (\xi \otimes \xi) &=
\lim_{n \to \infty} {\bf 1}_{\vr_n > 0} \frac{ \vm_n \otimes \vm_n }{\vr_n} :(\xi \otimes \xi)
- {\bf 1}_{\vr > 0} \frac{ \vm \otimes \vm }{\vr}:(\xi \otimes \xi)\\
& =
\lim_{n \to \infty} \frac{|\vm_n \cdot \xi|^2}{\vr_n} - \frac{|\vm \cdot \xi|^2}{\vr} \ \mbox{in}\ \mathcal{D}'((0,T) \times B) 
\end{split}
\]
for any bounded ball $B \subset R^d$;
whence the desired conclusion follows from the weak lower semicontinuity of the convex function $[\vr,\vm]\mapsto \frac{|\vm\cdot\xi|^{2}}{\vr}$, $\xi\in R^{d}$. We conclude that
\[
\mathfrak{R}_v \in L^\infty_{w^{*}}(0,T; \mathcal{M}^+(R^d; R^{d \times d}_{\rm sym})).
\]

Finally, similarly to \eqref{A1}, we note that
\begin{equation} \label{A1A}
\begin{split}
&\int_0^T \int_{R^{d}}\psi(t) \bfphi(x): \D \mathfrak{R}_v(t) \dt \\
&\qquad = \lim_{n \to \infty}
\int_0^T \intO{ \psi(t) \bfphi(x) : \left( {\bf 1}_{\vr_n > 0} \frac{\vm_n \otimes \vm_n}{\vr_n} - {\bf 1}_{\vr > 0}\frac{\vm \otimes \vm}{\vr} \right) } \dt \\ &\quad\mbox{for any}\
\psi \in L^1(0,T), \ \bfphi \in C_c(R^d; R^{d \times d})).
\end{split}
\end{equation}

\subsubsection{Total defect}

We introduce the \emph{total defect measure}
\begin{equation}\label{eq:d1}
\mathbb{D} \equiv \mathfrak{R}_v + (\gamma - 1) \mathfrak{R}_e \mathbb{I} \in
L^\infty_{w^{*}}(0,T; \mathcal{M}^+ (R^d; R^{d \times d}_{\rm sym} )),
\end{equation}
which describes the defect in the momentum equation.
Moreover, we get for the total energy
\begin{equation} \label{A3}
E(\vr_n, \vm_n|\vr_{\infty},\vm_{\infty})\to E(\vr, \vm|\vr_{\infty},\vm_{\infty}) 
+ \frac{1}{2} {\rm trace} [\mathfrak{R}_v] + \mathfrak{R}_e
\end{equation}
weakly-(*) in $L^\infty_{w^{*}}(0,T; \mathcal{M}^+(R^d; R^{d \times d}_{\rm sym}))$. In other words, we have a precise relation of the defect in the momentum equation and the defect of the energy. Finally, we get from \eqref{A3} that
\begin{equation*}% \label{A5}
\begin{split}
\int_0^T &\intO{\psi (t)\varphi(x)
\left( \frac{1}{2} \frac{|\vm_n|^2}{\vr_n} - \frac{1}{2} \frac{|\vm|^2}{\vr}  + P(\vr_n)  - P({\vr})
\right)  }\dt  \\ &\to
\int_0^T \psi (t)  \varphi(x) \,\D\left( \frac{1}{2} {\rm trace} [\mathfrak{R}_v(t)] + \mathfrak{R}_e(t)\right) \dt
\end{split}
\end{equation*}
for any $\psi \in L^1(0,T)$ and any $\varphi \in C_c(R^d)$.

\subsubsection{Bounded domain}
The above construction of the turbulent defect measure $\mathbb{D}$ as well as the proof of its properties can be carried out the same way on a  bounded domain $\Omega\subset R^{d}$, while using the dualities
$$
L^{1}(0,T;C(\overline{\Omega}))^{*}\cong L^{\infty}_{w^{*}}(0,T;\mathcal{M}(\overline{\Omega})) \ \mbox{and}\ L^{1}(0,T;C_{0}(\overline{\Omega};R^{d\times d}))^{*}\cong L^{\infty}_{w^{*}}(0,T;\mathcal{M}(\overline{\Omega};R^{d\times d})),
$$
respectively, where $\mathcal{M}(\overline{\Omega})$ is  the set of bounded Borel measures on $\overline{\Omega}$ (and similarly for the matrix--valued case).

\subsection{Asymptotic limit}

Using \eqref{A1}, \eqref{A1A} we may perform the asymptotic limit in the momentum equation \eqref{D17} obtaining
\begin{equation} \label{S1}
\begin{split}
\int_0^T &\intO{ \Big[ \partial_t \psi \vm \cdot \bfphi   + \psi {\bf 1}_{\vr > 0} \frac{\vm \otimes \vm}{\vr} : \Grad \bfphi
+ \psi p(\vr) \Div \bfphi  \Big] } \dt \\ &= -
\int_0^T \psi \Big[ \Grad \bfphi :\D\mathfrak{R}_v(t) + (\gamma - 1)  \Div \bfphi\, \D\mathfrak{R}_e(t) \Big]\dt
\\
&\mbox{for any}\ \psi \in C^1_c(0,T), \ \bfphi \in C^1_c(R^d; R^d).
\end{split}
\end{equation}

Thus, if the limit is a weak solution of the Euler system, then the left hand side of \eqref{S1} vanishes. Hence, in view of the definition of  the total defect measure \eqref{eq:d1}, we obtain
\begin{equation*}% \label{S2}
\int_{R^d} \Grad \bfphi : \D \mathbb{D}(t) = 0 \ \mbox{for any}\ \varphi \in C^1_c(R^d; R^d)\ \mbox{for a.a.}\ t \in (0,T)
\end{equation*}
which is nothing else than \eqref{def1}. 

\subsubsection{Equation $\Div \mathbb{D} = 0$ in $R^d$}

The following result, which can be regarded as a version of Liouville's theorem, is crucial in the proof of Theorem \ref{DT2}.
\begin{Proposition} \label{SP1}
Let $\mathbb{D} \in \mathcal{M}^+ (R^d; R^{d \times d}_{\rm sym})$ satisfy
\begin{equation} \label{S4}
\int_{R^d} \Grad \bfphi : \D \mathbb{D} = 0 \ \mbox{for any}\ \varphi \in C^1_c(R^d; R^d).
\end{equation}

Then $\mathbb{D} \equiv 0$.
\end{Proposition}

\begin{Remark} \label{SR1}

The assumption that the matrix $\mathbb{D}$ is positive semidefinite (or alternatively negative semidefinite, as a matter of fact), is absolutely essential. Indeed, DeLellis and Sz\' ekelyhidi in their
proof of the so-called oscillatory lemma in \cite{DelSze3} showed the existence of infinitely many smooth fields
$\mathbb{D} \in C^\infty_c(R^d; R^{d \times d}_{\rm sym})$ satisfying $\Div \mathbb{D} = 0$.

\end{Remark}

\begin{proof}[Proof of Proposition \ref{SP1}]
The proof relies on the extension of \eqref{S4} to all functions $\bfphi \in C^1(R^d; R^d)$ with
$\Grad \bfphi \in L^\infty (R^d; R^{d \times d})$, which is possible since $\mathbb{D}$ is a finite measure. This then permits to test \eqref{S4} by linear  functions $\bfphi$ and the conclusion follows from the positive semidefinitness of  $\mathbb{D}$.

To this end, let us consider a sequence of cut--off functions
\[
\psi_n \in \DC(R^d), \ 0 \leq \psi \leq 1,\ \psi_n(x) = 1
\ \mbox{for}\ |x| \leq n, \ \psi_n(x) = 0 \ \mbox{for}\ |x| \geq 2n, \ |\Grad \psi| \aleq \frac{1}{n}
\]
uniformly for $n \to \infty$.

For $\bfphi \in C^1(R^d; R^d)$, with $\Grad \bfphi \in L^\infty (R^d; R^{d \times d})$, we have
\[
|\bfphi(x)| \aleq (1 + n) \ \mbox{for all}\ x \in \mathrm{supp}\,\psi_{n};
\]
whence
\[
\begin{split}
0 &= \int_{R^d} \Grad (\psi_{n} \bfphi) : \D \mathbb{D} =
\int_{R^d} \psi_{n} \Grad \bfphi : \D \mathbb{D}
+ \int_{R^d} (\Grad \psi_{n}) \otimes \bfphi : \D \mathbb{D}\\
&= \int_{|x| \leq  n} \Grad \bfphi : \D \mathbb{D} +
\int_{n < |x| < 2n} \psi_{n} \Grad \bfphi : \D \mathbb{D} +
\int_{n < |x| < 2n}(\Grad \psi_{n}) \otimes \bfphi   : \D \mathbb{D}
\end{split}
\]
Seeing that
\[
| \psi_{n} \Grad \bfphi (x) | + | (\Grad \psi_{n}) \otimes \bfphi | \aleq 1 \ \mbox{whenever}\ n \leq |x| \leq 2n
\]
we may use the fact that $\mathbb{D}$ is a finite (signed) measure together with Lebesgue's dominated convergence theorem to  let $n \to \infty$ and conclude that
\begin{equation} \label{S5}
\int_{R^d} \Grad \bfphi : \D \mathbb{D} = 0 \ \mbox{for any}\ \varphi \in C^1(R^d; R^d),\ \Grad \bfphi \in L^\infty (R^d; R^{d \times d}).
\end{equation}

Finally, given a vector $\xi \in R^d$, we may use
\[
\bfphi(x) = \xi (\xi \cdot x)
\]
as a test function in \eqref{S5} to obtain
\[
\int_{R^d} (\xi \otimes \xi) : \D \mathbb{D} = 0 \ \mbox{for any}\ \xi \in R^d.
\]
As $\mathbb{D}$ is positive semidefinite in the sense of \eqref{eq:eq2}, i.e. $(\xi \otimes \xi) : \mathbb{D}$ is a non--negative finite measure on $R^{d}$, this yields $(\xi \otimes \xi) : \mathbb{D} = 0$ for any $\xi \in R^d$. 
Thus for any
$g \in C_b(R^d)$, $g \geq 0$, and the matrix
$
\int_{R^d} g \,\D \mathbb{D}
$
is positive semidefinite and we may infer
\[
\int_{R^d} g\, \D D_{i,j} = 0 \ \mbox{for any}\ i,j.
\]
As $g$ was arbitrary, this yields the desired conclusion $\mathbb{D} \equiv 0$.
\end{proof}

\subsubsection{Equation $\Div \mathbb{D} = 0$ in a bounded domain}

A trivial example of a constant--valued matrix shows that Proposition \ref{SP1} does not hold if $R^d$ is replaced by a bounded
domain $\Omega$ unless some extra restrictions are imposed. In addition to the hypotheses of Proposition \ref{SP1}, we shall assume that $\mathbb{D}$ vanishes sufficiently fast near the boundary $\partial \Omega$.

\begin{Proposition} \label{SP2}
Let $\Omega \subset R^d$ be a bounded domain.
Let $\mathbb{D} \in \mathcal{M}^+ (\Ov{\Omega}; R^{d \times d}_{\rm sym})$ satisfying
\begin{equation} \label{S6}
\int_{R^d} \Grad \bfphi : \D \mathbb{D} = 0 \ \mbox{for any}\ \varphi \in C^1_c(\Omega; R^d),
\end{equation}
and
\begin{equation} \label{S7}
\frac{1}{\delta} \int_{ \{x\in\Omega;{\rm dist}[x, \partial \Omega] \leq \delta\} } \D ({\rm trace})[\mathbb{D}]\to 0
\ \mbox{as}\ \delta \to 0.
\end{equation}

Then $\mathbb{D} \equiv 0$.

\end{Proposition}

\begin{proof}

Similarly to the proof of Proposition \ref{SP1}, it is enough to show that \eqref{S6} can be extended to a suitable function
$\varphi \in C^1(\Ov{\Omega}; R^d)$, whose gradient is constant.

It is a routine matter, cf. e.g. Galdi \cite{GAL}, to construct a sequence of cut--off functions $\psi_n$ enjoying the following properties:
\[
\psi_n \in C^1_c(\Omega), \ 0 \leq \psi_n \leq 1, \ \psi_n (x) = 1 \ \mbox{whenever}
\ {\rm dist}[x, \partial \Omega] > \frac{1}{n},\ |\Grad \psi_n| \aleq n.
\]
%To this end, let
%\[
%\Omega_{n}=\{x\in \Omega;\,{\rm dist}[x,\partial\Omega]\geq \frac\}
%\]
Thus, plugging $\psi_n \bfphi$, $\bfphi \in C^1(\Ov{\Omega}; R^d)$ in \eqref{S6} we get
\[
\begin{split}
0 &= \int_{\Omega} \Grad (\psi_{n} \bfphi) : \D \mathbb{D} =
\int_{\Omega} \psi_{n} \Grad \bfphi : \D \mathbb{D}
+ \int_{\Omega} (\Grad \psi_{n}) \otimes \bfphi : \D \mathbb{D}\\
&= \int_{{\rm dist}[x, \partial {\Omega}] > \frac{1}{n}} \Grad \bfphi : \D \mathbb{D} +
\int_{{\rm dist}[x, \partial {\Omega}] \leq \frac{1}{n}} \psi_{n} \Grad \bfphi : \D \mathbb{D} +
\int_{{\rm dist}[x, \partial {\Omega}] \leq \frac{1}{n}}(\Grad \psi_{n}) \otimes \bfphi   : \D \mathbb{D}
\end{split}
\]
Now, we observe that
\[
|\psi_{n} \Grad \bfphi(x)| + |(\Grad \psi_{n}) \otimes \bfphi(x)| \aleq n \ \mbox{whenever}\ {\rm dist}[x, \partial {\Omega}] \leq \frac{1}{n},
\]
which due to \eqref{S7} allows to pass to the limit as $n\to\infty$ in the second and the third term on the right hand side. The convergence of the first term follows from the fact that by \eqref{S7} the defect vanishes on the boundary, i.e.
\[
\int_{\partial \Omega} \D |\mathbb{D}| = 0,
\]
  and in the interior of $\Omega$ we have pointwise convergence of the corresponding integrand.
\end{proof}

\subsection{Strong convergence}

Applying Proposition \ref{SP1} in the situation of Theorem \ref{DT2} we obtain that $\mathfrak{R}_v\equiv 0$ and $ \mathfrak{R}_e \equiv 0$. In accordance with \eqref{A3}, this yields
\begin{equation} \label{C1b}
\begin{split}
E(\vr_n, \vm_n|\vr_{\infty},\vm_{\infty}) \to E(\vr, \vm|\vr_{\infty},\vm_{\infty}) 
\end{split}
\end{equation}
weakly-(*) in $L^\infty_{w^{*}}(0,T; \mathcal{M}^+(R^d))$. We show that this implies the strong convergence claimed in
Theorem \ref{DT2}.

First, we recall that both kinetic and internal energy are convex functions of
the density and the momentum so from \eqref{C1b} we obtain
\[ \int_0^T \int_{R^d} \left[ \frac{| \tmmathbf{m}_{n}|^2}{\vr_n} - 2 \vm_n \cdot \bu_\infty + \vr_n |\bu_\infty|^2 \right]
    \,\mathd x\, \mathd t \rightarrow \int_0^T
   \int_{R^d} \left[ \frac{| \tmmathbf{m} |^2}{\varrho} - 2 \vm \cdot \bu_\infty + \vr |\bu_\infty|^2 \right] \,\mathd x
   \,\mathd t, \]
\[ \int_0^T \int_{R^d} P (\varrho_{n}) - P' ({\varrho}_\infty)
   (\varrho_{n} - {\varrho}_\infty) - P ({\varrho}_\infty)\, \mathd x \,\mathd
   t \rightarrow \int_0^T \int_{R^d} P (\varrho) - P' ({\varrho}_\infty) (\varrho
   - {\varrho}_\infty) - P ({\varrho}_\infty)\, \mathd x \,\mathd t. \]
   Moreover, we may apply convexity again to deduce that
   \begin{equation}\label{eq:M2}
 \int_{B} \left[ \frac{| \tmmathbf{m}_{n}|^2}{\vr_n} - 2 \vm_n \cdot \bu_\infty + \vr_n |\bu_\infty|^2 \right]
    \,\mathd x\, \mathd t \rightarrow 
   \int_{B} \left[ \frac{| \tmmathbf{m} |^2}{\varrho} - 2 \vm \cdot \bu_\infty + \vr |\bu_\infty|^2 \right] \,\mathd x
   \,\mathd t,
   \end{equation}
\[\int_{B}P (\varrho_{n}) - P' ({\varrho}_\infty)
   (\varrho_{n} - {\varrho}_\infty) - P ({\varrho}_\infty)\, \mathd x \,\mathd
   t \rightarrow \int_{B} P (\varrho) - P' ({\varrho}_\infty) (\varrho
   - {\varrho}_\infty) - P ({\varrho_{\infty}})\, \mathd x \,\mathd t ,\]
for every Borel set $B \subset [0,T] \times R^d$.

Accordingly, choosing $B=[0,T]\times K$ for a compact set $K\subset R^{d}$, we obtain the convergence of the  norms of $\vr_{n}$ in $L^{\gamma}([0,T]\times K)$,  hence the strong convergence
\[
\vr_n \to \vr \ \mbox{ in}\ L^\gamma([0,T]\times K).
\]
The strong convergence on the full space $[0,T]\times R^{d}$ now follows by a tightness argument. Indeed, due to the weak convergence of the measures in \eqref{C1}, Prokhorov's theorem yields their tightness. In particular,  for  a given $\varepsilon>0$ there exists a compact set $K\subset R^{d}$ such that
\[
\sup_{n=1,2,\dots}\int_{0}^{T}\int_{K^{c}} P(\vr_n) - P'({\vr_{\infty}})(\vr_n - {\vr_{\infty}}) - P({\vr_{\infty}})\dx\dt<\varepsilon,
\]
\[
\int_{0}^{T}\int_{K^{c}} P(\vr) - P'({\vr_{\infty}})(\vr - {\vr_{\infty}}) - P({\vr_{\infty}})\dx\dt<\varepsilon.
\]
Finally, we write
\[
\begin{split}
&\|\vr_{n}-\vr\|_{(L^{\gamma}+L^{2})([0,T]\times R^{d})}\leq \|\vr_{n}-\vr\|_{(L^{\gamma}+L^{2})([0,T]\times K)}+\|\vr_{n}-\vr_{\infty}\|_{(L^{\gamma}+L^{2})([0,T]\times K^{c})}\\
&\qquad\qquad\qquad+\|\vr-\vr_{\infty}\|_{(L^{\gamma}+L^{2})([0,T]\times K^{c})}\\
&\qquad\leq\|\vr_{n}-\vr\|_{(L^{\gamma}+L^{2})([0,T]\times K)}+\int_{0}^{T}\int_{K^{c}} P(\vr_n) - P'({\vr_{\infty}})(\vr_n - {\vr_{\infty}}) - P({\vr_{\infty}})\dx\dt\\
&\qquad\qquad\qquad+\int_{0}^{T}\int_{K^{c}} P(\vr) - P'({\vr_{\infty}})(\vr - {\vr_{\infty}}) - P({\vr_{\infty}})\dx\dt,
\end{split}
\]
where the first term converges to zero as $n\to\infty$ whereas the second as well as the third term is small uniformly in $n$.

Let us now  establish the strong convergence of the momenta on $[0,T]\times R^{d}$. To this end, we recall that by the energy bounds it holds (up to a subsequence)
\[
\tmmathbf{h}_n \equiv \frac{\vm_n}{\sqrt{\vr_n}} \to \tmmathbf{h} \ \mbox{weakly in}\ L^{2}([0,T]\times B; R^d)
\]
for some $\tmmathbf{h}\in L^{2}([0,T]\times B; R^d)$,
and by \eqref{eq:M1}
\[
\vm_n \to \vm \ \mbox{weakly in}\ (L^{\frac{2 \gamma}{\gamma + 1}})([0,T]\times B; R^d)
\]
for any bounded ball $B \subset R^d$.
We shall show that
\[
\tmmathbf{h} ={\bf 1}_{\vr>0} \frac{\vm}{\sqrt{\vr}} \ \mbox{a.a. in}\ [0,T] \times B.
\]
Combining the weak convergence of $\tmmathbf{h}_{n}$ with the strong convergence of $\vr_{n}$ and the weak convergence of $\vm_{n}$ we obtain
\[
\sqrt{\vr_n} \tmmathbf{h}_n = \vm_n \to \vc{m} = \sqrt{\vr} \tmmathbf{h} \ \mbox{weakly in}\ L^{1}([0,T]\times B; R^d);
\]
whence it is enough to prove that $\tmmathbf{h} = 0$ whenever $\vr= 0$. By weak lower semicontinuity of the $L^{2}$-norm together with \eqref{eq:M2}, we obtain
\[
\int_{ \vr < \delta } {\bf 1}_B |\tmmathbf{h}|^2 \dx\dt \leq \lim_{n \to \infty}  \int_{ \vr < \delta }{\bf 1}_B \frac{|\vm_n|^2}{\vr_n} \dx\dt
= \int_{ \vr < \delta } {\bf 1}_B \frac{|\vm|^2}{\vr} \dx\dt .
\]
Now, it is enough to observe that in the limit $\delta\to0$, the left hand side converges to
$$
\int_{ \vr =0 } {\bf 1}_B |\tmmathbf{h}|^2 \dx\dt,
$$
whereas the right hand side vanishes, since due to the integrability of  the kinetic energy $\frac{|\vm|^{2}}{\vr}$ it holds that the set, where $\varrho=0$ and $\vm\neq 0$, is of zero Lebesgue measure. Thus $\tmmathbf{h} = 0$ whenever $\vr= 0$.

To summarize, we have shown that
\[ \frac{\tmmathbf{m}_{n}}{\sqrt{\varrho_{n}}} \rightarrow
   \tmmathbf{1}_{\varrho > 0} \frac{\tmmathbf{m}}{\sqrt{\varrho}}\ \mbox{weakly in}\ L^{2}([0,T]\times B; R^d)\]
and hence strongly due to \eqref{eq:M2}, which  implies the strong convergence
\[ \tmmathbf{m}_{n} = \sqrt{\varrho_{n}}
   \frac{\tmmathbf{m}_{n}}{\sqrt{\varrho_{n}}} \rightarrow
   \tmmathbf{m} \ \mbox{in}\  L^{\frac{2 \gamma }{\gamma + 1}}([0,T]\times B; R^d) . \]
   Finally, a tightness argument as for the density above implies the strong convergence
   $$
   (\vm-\vm_{\infty})\to (\vm-\vm_{\infty})\ \mbox{in}\  (L^{\frac{2\gamma}{\gamma+1}}+L^{2})([0,T]\times R^{d};R^{d}).
   $$

   The strong convergence of the densities and the momenta from Theorem~\ref{DT2} now follows immediately from the fact that uniformly in $n$
$$
(\vr_{n}- \vr_\infty)\in L^{\infty}(0,T;(L^{\gamma}+L^{2})(R^{d})),\quad (\vm_{n} - \vm_\infty) \in L^{\infty}(0,T;(L^{\frac{2\gamma}{\gamma+1}}+L^{2})(R^{d};R^{d})),
$$
due to the energy bound. The convergence (up to a subsequence) of the energies in $L^{1}$ is then a consequence of the strong convergence of $\frac{|\vm_{n}|}{\sqrt{\vr_{n}}}$ and $\vr_{n}$ together with \eqref{C1} and Vitali's theorem.
This completes the proof of Theorem \ref{DT2}.

\section{Concluding remarks}
\label{DD}

We conclude the paper by a short discussion on possible extensions of Theorem \ref{DT2}.
As indicated in Proposition \ref{SP2}, the conclusion of Theorem \ref{DT2} remains valid on bounded Lipschitz domains provided 
some extra assumptions about the behavior of the consistent approximation near the boundary is assumed. The relevant result can be stated as follows.
 
\begin{Theorem}[Asymptotic limit of consistent approximation in bounded domains] \label{DT3}
Let $\{ \vr_n, \vm_n \}_{n=1}^\infty$
be a consistent approximation of the isentropic Euler system in $(0,T) \times \Omega$ in the sense of Definition~\ref{DD4}, 
where $\Omega \subset R^d$ is a bounded Lipschitz domain, and where we have set $\vr_\infty = \vu_\infty = 0$ in the relative energy.

Then there exists a subsequence (not relabeled for simplicity) enjoying the following properties:
\[
\begin{split}
\vr_n &\to \vr \ \mbox{weakly-(*) in}\ L^\infty(0,T; L^\gamma(\Omega)),\\
\vm_n &\to \vm \ \mbox{weakly-(*) in}\ L^\infty(0,T; L^{\frac{2 \gamma}{\gamma + 1}}(\Omega; R^d)).
\end{split}
\]
Suppose, in addition, that 
\[
\limsup_{n \to \infty} \int_{ x \in \Omega; {\rm dist}[x, \partial \Omega] < \delta} 
\left( E(\vr_n, \vm_n) - E(\vr, \vm) \right)(\tau, \cdot) \dx 
\]
is of order $o(\delta)$ as $\delta \to 0$.
Then if $[\vr, \vm]$ is an admissible weak solution to the isentropic Euler system, then
\begin{equation} \label{DDD1}
\int_0^T \intO{ {E} \left(\vr_n, \vm_n \ \Big| \vr , \vm \right) } \dt \to 0,
\end{equation}
in particular,
\[
\begin{split}
\vr_n &\to \vr \ \mbox{in}\ L^q(0,T; L^\gamma (\Omega)),\\
\vm_n &\to \vm \ \mbox{in}\ L^q(0,T; L^{\frac{2 \gamma}{\gamma + 1}}(\Omega; R^d)),
\end{split}
\]
for any $1 \leq q < \infty$. Thus, for a suitable subsequence,
\[
\vr_n \to \vr, \ \vm_n \to \vm \ \mbox{a.a. in}\ (0,T) \times \Omega.
\]
\end{Theorem}

The hypothesis \eqref{DDD1} is satisfied if, for instance, 
\[
\lim_{n \to \infty} \left\| E(\vr_n, \vm_n) - E(\vr, \vm) \right\|_{L^1((0,T) \times \mathcal{U} )} = 0,
\]
where $\mathcal{U}$ is an open neighborhood of $\partial \Omega$.

%\bibliography{citace}
%\bibliographystyle{plain}

\def\cprime{$'$} \def\ocirc#1{\ifmmode\setbox0=\hbox{$#1$}\dimen0=\ht0
  \advance\dimen0 by1pt\rlap{\hbox to\wd0{\hss\raise\dimen0
  \hbox{\hskip.2em$\scriptscriptstyle\circ$}\hss}}#1\else {\accent"17 #1}\fi}

\end{document}